\documentclass[12pt, a4paper]{amsart}

\makeatletter
\@namedef{subjclassname@2010}{%
  \textup{2010} Mathematics Subject Classification}
\makeatother

 \usepackage{verbatim}
\usepackage{graphicx}
\usepackage{amsmath}            
\usepackage{amsfonts}           
\usepackage{amssymb}            
\usepackage{xcolor}             
\usepackage{hyperref}

\newtheorem{thm}{Theorem}[section]
 
 \newtheorem{lem}[thm]{Lemma}
 \newtheorem{prop}[thm]{Proposition}
 \newtheorem{oprob}[thm]{Open Problem}
 \theoremstyle{definition}
 \newtheorem{defn}[thm]{Definition}
 \theoremstyle{remark}
 \newtheorem{rem}[thm]{Remark}
 \newtheorem{ex}[thm]{Example}
 \numberwithin{equation}{section}

\newcommand{\psh}{plurisubharmonic }
\newcommand{\om}{\ensuremath{\Omega}}
\newcommand{\cn}{\ensuremath{\mathbb{C}^n}}

\newcommand{\p}{\ensuremath{\partial}}
\newcommand{\eps}{\ensuremath{\varepsilon}}

\newcommand{\mam}{Monge-{A}mp\`ere measure }

\DeclareMathOperator{\supp}{supp}
\DeclareMathOperator{\e}{e}
\DeclareMathOperator{\inte}{int}

\begin{document}
 \baselineskip=17pt
\title{Supports of Weighted Equilibrium Measures and Examples}
\address{Johns Hopkins University, Baltimore, MD,21218 USA}
\email{malan@math.jhu.edu}

\author{Muhammed Al\.I    ALAN}

\thanks{This paper is part of the thesis of the author.}


\begin{abstract}
We analyze the supports of  weighted equilibrium measures in \cn.
We  give explicit examples of families of compact sets which arise as the support
of a weighted equilibrium measure for some admissible weight $w$.
These examples also give new constructions of plurisubharmonic functions in the Lelong class.
We also include a list of open problems on the support of extremal measures which are related to solutions of Monge-{A}mp\`ere equations.
\end{abstract}
\maketitle
\section{Introduction and Background}

Determining  the support $S_w$  of the \mam of a weighted  extremal function $V_{K,Q}$ is  important for many reasons.  Firstly, the
domination principle (Theorem \ref{DominationPrinciple}) shows the importance of the support of the \mam of a weighted
extremal function. Namely,  we have $V_{K,Q}=V_{S_w,{Q|{S_w}}}$.
 Secondly, for a continuous
continuous weight $w$ on a closed set $K$, the weighted supremum
norm of a polynomial $p_d$ of degree less than or equal to $d$ is attained on the support of the  weighted extremal
measure i.e.,  $\|w^d p_d\|_K=\|w^d p_d\|_{S_w}$.

Next, a weighted approximation is possible  only on the support of
the \mam of the weighted  extremal function. Namely, if $f \in \mathcal{C}(K)$, then there exists a weighted sequence
$(w^d p_d)_{d=1}^\infty $ converging uniformly to $f$ on $K$ only if $f(z)=0$ for all $z\in K\setminus S_w$.
Some references for the weighted approximation are \cite{CallaghanThesis},  \cite{Saff-Totik} and \cite{TotikVarju}.

In one variable, there is a notion of logarithmic capacity associated to a compact set $K$.
In order to estimate the weighted  analogue of the capacity with respect to an admissible weight $w$ on a compact set $K$ numerically, it is useful to determine the support of the
weighted extremal measure.  Using the fact that the weighted capacity of a closed
set is equal to the weighted capacity of the support of the weighted extremal measure,
one may reduce the  calculation time remarkably, e.g., compare Table 4 and Table 5 in \cite{RajonRansfordRostand}. See \cite{RajonRansfordRostand},
  \cite{Saff-Totik} for
 further discussion of weighted capacities.

This paper has two motivations. The first motivation is to further analyze the supports of weighted equilibrium measures and to extend the
results from $n=1$ and from the unweighted setting to the weighted setting in higher dimensions.
The second one is to provide families of examples of compact sets which arise as supports of weighted equilibrium measures.
Furthermore, some of these examples
give  new ways of constructing  plurisubharmonic functions in the Lelong class.

In the second section, we recall some of the basic definitions and necessary facts from weighted and classical pluripotential theories.
We also demonstrate by a counterexample that in the weighted setting, unlike the
unweighted setting, the weighted extremal measure and the weighted relative extremal measures are not necessarily mutually absolutely continuous.
In the third section, we give several families of examples of compact sets that are supports of weighted extremal measures. In particular, to
construct the radial extremal functions we survey  Persson's representation
of radial plurisubharmonic  functions in terms of their Monge-{A}mp\`ere measures. In Theorem \ref{SupportStrPsConvex}, which is the main theorem and the deepest result, we  construct extremal
functions whose Monge-{A}mp\`ere measures are supported on  the closure of strictly
pseudoconvex domains.
In the last section, we discuss  some open problems related to the subject.

\section{Weighted Pluripotential Theory}\label{WeightedPT}

We recall basic definitions and facts from  weighted and classical pluripotential
theories. We refer to Saff and Totik's book
\cite{Saff-Totik} for  $n=1$ and Thomas Bloom's Appendix B of
\cite{Saff-Totik} for $n>1$.

Let $K$ be a closed subset of \cn.  A function $w:K\to [0,\infty)$ is called an
\textbf{admissible weight function} on $K$ if
\begin{itemize}
\item[i)] $w$ is upper semicontinuous.
\item[ii)] The set $\{z \in K \mid  w(z) >0 \}$  is not pluripolar.
\item[iii)] If $K$ is unbounded,   $|z|w(z)\to 0 $ as  $|z|\to \infty,\,z\in K$.
\end{itemize}
We define  $Q =  - \log w $. We  use $Q$ and $w$ interchangeably.

The \textbf{weighted  extremal function} of $K$ with respect to $Q$  is defined as
\begin{equation}
V_{K,Q}(z):= \sup\left\{ u(z) \mid  u\in L, \, u\leq Q \text{ on } K\right\},
\end{equation}
where the Lelong class $L$ is defined as
\begin{equation}
 L:=\{ u\mid  u \in PSH(\cn), \, u(z) \leq \log^+ |z| +C \},
\end{equation}
where $C$ depends on $u$.  Here the set of all plurisubharmonic functions on
a domain $\om$ is denoted by $PSH(\om)$. If $Q=0$, then we call it the \textbf{(unweighted) extremal function} of $K$  and denote it by $V_{K}$.
A compact set $K$ is called \textbf{regular} if $V_{K}$ is continuous. If $K\cap \overline{B(z,r)}$ is regular for all $z\in K$ and $r>0$,
the set $K$
is called \textbf{locally regular}. Here, $B(z,r)$ denotes the open ball of radius $r$ and center $z$.

We note  that the  \textbf{upper semicontinuous regularization} of a function $v$ is defined by $v^\ast(z):=\limsup\limits_{w\to z}v(w)$ and it is well known that the upper semicontinuous regularization of $V_{K,Q}$ is plurisubharmonic and in $L^+$ where
$$L^+ := \{u \in L \mid  \log^+ |z| +C \leq  u(z) \},$$where $C$ depends on $u$.

We recall that  $dd^c u= 2 i \p \bar{\p} u $ and $(dd^c u)^n$   is the \textbf{complex Monge-{A}mp\`ere operator} defined  by $(dd^c
u)^n=dd^c u \wedge \dots \wedge dd^c u$ for plurisubharmonic functions which are $\mathcal{C}^2$.
If $u$ is   a locally bounded plurisubharmonic function, then  $(dd^c u)^n$ is defined as a positive measure. See \cite{Klimek} for the details. It is also well known that $\int_ {\cn} (dd^c u)^n=(2\pi)^n$ for
all $u\in L^+$.

A set $E$ is called \textbf{pluripolar} if $E\subset \{z\in \cn \mid
u(z)=-\infty\}$ for some plurisubharmonic function $u$. If a
property holds everywhere except on a pluripolar set we will say that
the property holds \textbf{quasi everywhere}.
It is a well-known fact that $V_{K,Q}=V_{K,Q}^\ast$ quasi everywhere.

We denote the support of  $(dd^cV_{K,Q}^\ast)^n$ by  $S_w$.
The following lemma is very useful to determine the supports of Monge-{A}mp\`ere measures.

\begin{lem}\label{SuppprtOfExtremalMeasureSubsetSw} \cite[Appendix B, Lemma 2.3]{Saff-Totik}
Let $S_w^\ast:= \{z\in \cn \mid  V_{K,Q}^\ast(z)\geq Q(z) \}$.
Then, we have $S_w\subset S_w^\ast$.
\end{lem}

The following theorem  is a very important tool in pluripotential theory and we will use it frequently to determine
weighted extremal functions.

\begin{thm}\label{DominationPrinciple}\textbf{(Domination Principle.)} \cite[Lemma 6.5]{BedfortTaylorPSHFunctionsLogarithmicSingularities} If $u\in L, \, v\in L^+$ and if $u\leq v$ holds  almost everywhere with respect to $(dd^c v)^n$, then $u\leq v \text{ on } \cn$.
\end{thm}

\begin{defn}
A plurisubharmonic function $u$ on an open set $\Omega $ is called \textbf{maximal}  if for any
relatively compact open subset $\omega$ of $\Omega $ and any upper semicontinuous
function $v$ defined on $\overline{\omega}$ which is plurisubharmonic on $\omega$ such that
$v \leq u $\; on the boundary of $\omega$, then $v \leq u $ in $\omega$.

\end{defn}

The following theorem characterizes  maximal plurisubharmonic functions in terms of their Monge-{A}mp\`ere measures.

\begin{thm}\label{maximal}\cite[Theorem 4.4.2]{Klimek}
Let \om\, be an open subset of  \,\cn\, and $u$ be a locally bounded plurisubharmonic function on \om.
Then,  $(dd^{c}u)^{n}=0$ if and only if  $u$ is maximal.
\end{thm}

The following theorem is called \textbf{Poisson Modification} and it is used to modify a locally bounded plurisubharmonic function
in a ball to obtain another locally bounded plurisubharmonic function such that the new function is maximal in the ball
and equal to the original function outside the ball.

\begin{thm}\label{Poisson Modification}\cite[Appendix B, Theorem 1.3]{Saff-Totik}
Let $u\in PSH(\om)\cap  L^\infty_{loc}(\om)$. Let $z_0 \in \om$ and  $R>0$ such that   $\overline{B(z_0,R)}\subset
\om$. Then there exists $\tilde{u}\in PSH(\om)\cap L^\infty_{loc}(\om)$ such that
\begin{itemize}
  \item [i)] $(dd^c \tilde{u})^n =0 $ on $B(z_0,R)$;
  \item [ii)] $\tilde{u}=u$ on $\om \setminus B(z_0,R)$;
  \item [iii)] $\tilde{u}\geq u$ on $B(z_0,R)$.
\end{itemize}
\end{thm}

We recall the definition of weighted extremal functions relative to an open subset of \cn.
\begin{defn}\label{extremal}
Let \om\, be a bounded domain in \cn\, and $K$ be a compact subset of \om. Let $Q : K \to (-\infty,0]$ be a function   on
$K$. The \textbf{weighted relative extremal function} of $K$ with respect to \om  \,    and weight $Q$ is defined as
\begin{equation}\label{wre}
U_{K,Q,\om}(z):= \sup\{u(z)\mid  u \in PSH(\om),\, u<0 \text{ on } \om ,\, u\leq Q \text{ on } K\}.
\end{equation}
It is  well known that the  \textbf{regularized weighted relative extremal function}  of $K$ with respect to \om  \,    and weight $Q$ defined by
\begin{equation}\label{wre1}
U_{K,Q,\om}^\ast(z):= \limsup\limits_{w\to z}U_{K,Q,\om}(w)
\end{equation}
is plurisubharmonic and maximal out of $K$. Again, we have $U_{K,Q,\om}=U_{K,Q,\om}^\ast$ quasi everywhere.
\end{defn}
If $Q=-1$, then $U_{K,Q,\om}$ is called \textbf{(unweighted) relative extremal function} of $K$ with respect to \om, and we denote it by
$U_{K,\om}$.

Here we give an analogue of Lemma \ref{SuppprtOfExtremalMeasureSubsetSw} for  relative extremal functions.

\begin{prop}\label{RelativeExtremalMeasureonSW} Let $Q$ be an admissible weight on $K$ (i.e., $\{z\in K \mid Q(z)<0\} $ is not pluripolar), where  $K$ is
 a compact subset of a hyperconvex domain \om. Then,   $ (dd^cU_{K,Q,\om}^\ast)^n$ is supported on the set  $\{z\in K \mid  U_{K,Q,\om}^\ast(z)\geq Q(z)\}$.
\end{prop}
\begin{proof} Let $z_0$ be a point in $K$,  such  that  $U^\ast_{K,Q,\om}(z_0)<
Q(z_0)-\eps$, for some positive $\eps$. We will show that $U_{K,Q,\om}^\ast$ is maximal in a neighborhood $U$ of
$z_0$.

By the facts that $Q$ is lower semicontinuous and $U_{K,Q,\om}^\ast$ is upper semicontinuous, we have $\{z\in K \mid Q(z)>Q(z_0)-\eps/2\}$ is open  in $K$ relative to \om\,  and   $\{z\in \om\mid U_{K,Q,\om}^\ast(z)<
U_{K,Q,\om}^\ast(z_0)+\eps/2\}$ open. So we may find a ball of radius $r$ around $z_0$ such that
$$\sup\limits_{z\in B(z_0,r)}U_{K,Q,\om}^\ast(z)<\inf \limits_{z\in B(z_0,r)\cap K}Q(z).$$

Applying the Poisson modification (Theorem \ref{Poisson Modification}) to $U_{K,Q,\om}^\ast$ on $B(z_0,r)$, we can find  a
plurisubharmonic function $u$ such that $u\geq U_{K,Q,\om}^\ast$ on $B(z_0,r)$ and $u  =U_{K,Q,\om}^\ast$ on $\om\setminus
B(z_0,r)$, which is negative on \om, and   maximal on $B(z_0,r)$. Now $u$ is a competitor for the relative
 extremal function because $u(z)\leq \sup\limits_{z\in B(z_0,r)}U_{K,Q,\om}^\ast(z)<\inf
\limits_{z\in B(z_0,r)\cap K}Q(z)$ for all $z\in B(z_0,r)$. Hence, $u\equiv
U_{K,Q,\om}^\ast$. Therefore, we get $U_{K,Q,\om}^\ast$ is maximal in a
neighborhood  of $z_0$. Hence, $z_0$  is not in the support of $(dd^cU_{K,Q,\om}^\ast)^n$.
\end{proof}

Levenberg showed in  \cite{LevenbergMeasures} that $(dd^c V_{K}^\ast)^n$ { and }
$ (dd^cU_{K,\om}^\ast)^n $ are mutually absolutely continuous for a non-pluripolar compact set $K$.  However, unlike the
unweighted case,  $(dd^c V_{K,Q}^\ast)^n$ and $(dd^cU_{K,Q,\om}^\ast)^n $ are not necessarily absolutely continuous  in general.

 \begin{ex}\label{Example1} Let $K$ be the closed unit ball,  $\overline{B(0,1)}$. We define a continuous weight on $K$ by letting
  $ Q(z)=\max\left[\log |z|,-\frac{1}{2}\right]-1$. Let  $\om_1 :=B(0,2)$
  and $\om_2 :=B(0,\e)$. We have the weighted
  relative extremal functions
\begin{eqnarray}
\nonumber  U_{K,Q,\om_1}(z) &=& \max\left[ {\max\left
[\log|z|,-\frac{1}{2}\right]-1},
  \frac{\log|z|}{\log 2}-1\right], \\
\nonumber  U_{K,Q,\om_2}(z) &=& \max\left[\log |z|,-\frac{1}{2}\right]-1,
\end{eqnarray}
  and the weighted (global) extremal function
$$V_{K,Q}(z)= \max\left[\log |z|,\frac{1}{2}\right]-1.$$
Now the  weighted relative  extremal measure with respect to $\om_1$ is supported on two concentric spheres. Namely,
$\supp (dd^c U_{K,Q,\om_1})^n= \p B(0,1) \cup \p B(0,\e^{-1/2})$. Whereas,  $\supp (dd^c V_{K,Q})^n=\supp (dd^c
U_{K,Q,\om_2})^n=  \p B(0,\e^{-1/2})$. Therefore, $(dd^cU_{K,Q,\om_1})^n$ is not absolutely continuous with respect to
$(dd^cV_{K,Q})^n$ and  $(dd^cU_{K,Q,\om_2})^n$.
\end{ex}

\begin{ex}\label{Example2}  Let $K$ be the compact set $\overline{B(0,1)}$. We define a continuous weight on $K$ by letting
  $ Q:=\frac{1}{2}\max\left[\log |z|,-\frac{1}{2}\right]-\frac{1}{2}$. Let $\om$ be $B(0,\e)$.
  Under these conditions, we have the weighted
  relative extremal function
$$ U_{K,Q,\om}(z) = \frac{1}{2}\max\left[\log
|z|,-\frac{1}{2}\right]-\frac{1}{2},
$$
and the weighted (global) extremal function
$$V_{K,Q}(z)= \max\left[\frac{1}{2}\max\left[\log |z|,-\frac{1}{2}\right], \log |z|\right]-\frac{1}{2}.$$
In this case,  the global  (weighted)  extremal measure is supported on two concentric spheres. Namely,
 $\supp (dd^c V_{K,Q})^n= \p B(0,1) \cup \p B(0,\e^{-1/2})$. On the other hand, $\supp
 (dd^cU_{K,Q,\om})^n=\p B(0,\e^{-1/2})$. Therefore, $(dd^cV_{K,Q})^n$ is not
absolutely continuous with respect to $(dd^cU_{K,Q,\om})^n$.
\end{ex}

\section{Examples}
In classical pluripotential theory, the extremal measures, $(dd^c V_{K}^\ast)^n$ and $(dd^cU_{K,\om}^\ast)^n$, of a compact set $K$ are supported on the outer boundary of $K$.  In the weighted setting, this does not hold true
anymore. Many compact sets may arise as support of the  \mam of some weighted  extremal function. In particular,  by Theorem 4.1.1 of \cite{Saff-Totik}, for any compact set $K$ in the plane, which has positive logarithmic capacity at every point of $K$ (i.e., $C(K\cap B(z_0,r),\om)>0$ for every $z_0\in K$, for all $r>0$ and for some hyperconvex domain \om\, containing $K$),  there exists an
admissible weight on $K$ such that the support of the weighted extremal measure is $K$. Unfortunately, the proof  of the theorem utilizes  the notion of logarithmic potentials which is not available in higher dimensions. Thus, determining  the support of a weighted equilibrium measure and obtaining a similar result for higher dimensions are extremely  difficult. In one dimension,
Varju and Totik \cite{TotikVarju} found some necessary and sufficient conditions for
weights  on the unit circle  that the support of the weighted extremal measure is the whole unit circle.
In one variable, Benko, Damelin and Dragnev \cite{BenkoDamelinDragnev} also obtained some  sufficiency conditions for the same problem. They  also give explicit examples where the supports of the weighted equilibrium measures are the full circle or a finite union of arcs. Similar results for higher dimensions do not exists to our best knowledge.
In this
section, we show concrete examples of families  of compact sets which are supports of some  weighted extremal measures.
Some references for  supports of   weighted equilibrium measures are \cite{BenkoDamelinDragnev},
\cite{CallaghanThesis},  \cite{Saff-Totik} and \cite{TotikVarju}.

The next proposition  gives criteria for  the weight $Q$ such that the weighted extremal  measure is supported
on the boundary of $K$.

\begin{prop}\label{SuperharmonicExtremalMeasure} Let $K$ be a closed set and $Q$ is  an
admissible weight on $K$. Then   $S_w \subset \p  K $, if either of the following holds:
\begin{enumerate}
    \item $Q$ is maximal plurisubharmonic in the interior of $K$,
    \item $Q$ is superharmonic in the interior of $K$.
\end{enumerate}
\end{prop}

\begin{proof}
We give the proof of the first. The proof of second part is similar. In the superharmonic case we use the fact that a plurisubharmonic function which is also superharmonic is in fact pluriharmonic.

Let $Q$ be  superharmonic in the interior of $K$. Note that $Q$ is continuous in the interior of $K$ since it is both upper semicontinuous and lower semicontinuous.
Let $z_0 \in  \inte( K )$, i.e.,  there exists $r>0$ such that $B(z_0,r)\subset \inte( K )$. Since $Q$ is continuous $V^ \ast _{ K , Q}(z) \leq Q(z)$ for all $z\in B(z_0,r)$. By applying the  Poisson modification to the plurisubharmonic function $V^ \ast _{ K , Q} $ on the ball $B(z_0,r)$, we obtain $v\in L^+$ such that $v=V^ \ast _{ K , Q} $ on $\cn\setminus B(z_0,r)$ and $v$ is maximal on $B(z_0,r)$. Since $V^ \ast _{ K , Q} \leq Q$ on $ B(z_0,r)$, we have $v\leq Q$ on $B(z_0,r)$. Hence $v$ is a competitor for the weighted extremal function. Thus $v\equiv  V^ \ast _{ K , Q} $, giving that $V^ \ast _{ K , Q} $ is maximal on $ B(z_0,r)$.
\end{proof}

\begin{rem}
Note that similar proofs give the same results of  Proposition \ref{SuperharmonicExtremalMeasure} for the weighted relative extremal measure.
\end{rem} 

This was an extreme  case where the weighted extremal measure is supported on the boundary. Now another extreme case  is that
the closure of a strictly pseudoconvex domain can be obtained as the support of a weighted
extremal measure.

\begin{thm}\label{SupportStrPsConvex} Let \om \,   be   a strictly pseudoconvex domain in \cn\, with $\mathcal{C}^2$ boundary.
There exists an admissible  weight function  $Q$ on $K:= \overline{\om}$  such that $\supp (dd^{c} V_{K,Q}^\ast)^n =
\overline{\om}$.
\end{thm}

In order to prove Theorem  \ref{SupportStrPsConvex},  we need the following important  theorem on gluing
plurisubharmonic functions.

\begin{thm}\label{GluePshSmooth}  Let \om \,   be    a strictly pseudoconvex domain in \cn\, with $\mathcal{C}^2$  boundary.
Let $u\in \mathcal{C}^1(\overline{\om})\cap PSH(\om) $ and $v \in \mathcal{C}^1(\cn \setminus \om)\cap PSH(\cn \setminus \overline{\om}) $, such that $u=v$ on the boundary of \om. If the normal
derivatives satisfy   $\frac{\p u}{\p n} \leq \frac{\p v}{\p n}$ on the boundary of \om, then the function defined by  $$ V
:=\left\{\begin{array}{ccc}
u & \quad \, \text{ on } &  \om ;                       \\
v & \quad \, \text{ on }  & \cn \setminus \om   \\
\end{array}\right.
$$ is plurisubharmonic on \cn.
\end{thm}

Here $n$ is the outward unit normal and the normal derivative is $\frac{\p u}{\p n}= \nabla u \cdot n $. Blanchet and Gauthier
give the proof of the above theorem for subharmonic  functions  and domains with $\mathcal{C}^1$ boundary in \cite{GauthierBlanchet}. Moreover, they state this theorem for  subsolutions of any elliptic partial differential equation.
This theorem is true in a more general setting; here we give an elementary proof for the case of   plurisubharmonic functions in strictly pseudoconvex domains and we do not use the theory of distributions.

\begin{proof} Upper semicontinuity of  $V$ is trivial as it  is continuous.   Since  $u$   and $v$ are  \psh  on \om\,  and
$\cn \setminus\overline{\om}$ respectively, we have  $V$ is \psh on  $ \cn \setminus \p \om$. Thus, we need to show
that at every point $z\in \p\om$, we have $V$ is subharmonic  on each complex line passing through  $z$.

We fix  $z \in \p\om$. Since plurisubharmonicity is preserved under biholomorphic mappings, applying a translation and a rotation, we may assume that $z=0$ and the outer  normal, $n_1$, is the real vector $n_1=(1,0,\dots, 0)\in \mathbb{R}^{2n}$. Furthermore, we may take $u(0)=v(0)=0$.

By the  Narasimhan lemma (See  \cite{KrantzSCV} Lemma 3.2.3), there is a neighborhood $U$ of $0$ and a biholomorphic mapping of $U$, $\varphi : U \to \cn$, such that $\varphi(U\cap \om) $  is strictly convex. By applying such a biholomorphic mapping of a neighborhood of the origin, we may assume that  \om \,    is strictly convex around this point.
By the proof of the Narasimhan lemma, the Jacobian matrix of $\varphi$ at $z=0$ can be taken as  the identity matrix. Therefore, the unit normals, the normal derivatives and the inequality $\frac{\p u}{\p n}(0) \leq \frac{\p v}{\p n}(0)$  are preserved under the mapping.

Note that $u$ and $v$ can be extended to a neighborhood of the origin such that the extensions are  differentiable  at the origin. We  will  continue writing $u$ and $v$ for the extensions and note that our calculations are independent of the extensions. Now we will  show that $V$ satisfies the following integral
inequality on any complex line passing through $0$.
$$V(0) \leq \frac{1}{2\pi}\int\limits_0^{2\pi} V(0+ b\e^{i\theta}) d\theta $$ where $b = \rho
((a_1, b_1), \dots , (a_n, b_n))$ so that  $$b\e^{i\theta} = \rho ((a_1\cos(\theta),
b_1\sin(\theta)), \dots , (a_n\cos(\theta), b_n\sin(\theta)).$$
We may also assume that $a_1$ is positive. Otherwise, we may take $-b$ which gives the same complex line.

If the complex line intersects with \om \, only at $0$, i.e.,  the complex line lies entirely in the tangent space, then consider $z_j:=(x_j, 0,\dots, 0)\in \mathbb{R}^{2n}$ where $x_j>0$ such that $z_j\to 0$. Note that all $z_j \in \cn\setminus\overline{\om}$ and the complex line passing through $z_j$ which is  in the direction of $b$ lies completely in $\cn\setminus\overline{\om}$.

By plurisubharmonicity of $v$ on $\cn\setminus\overline{\om}$, we have
$$V(z_j)=v(z_j) \leq \frac{1}{2\pi}\int\limits_0^{2\pi} v(z_j + b\e^{i\theta}) d\theta =\frac{1}{2\pi}\int\limits_0^{2\pi} V(z_j + b\e^{i\theta}) d\theta. $$
As $z_j\to 0$, we have  $v(z_j) \to v(0)$. By uniform continuity of $v$ in a neighborhood of $0$, we have
$$\int\limits_0^{2\pi} V(z_j + b\e^{i\theta}) d\theta \to \int\limits_0^{2\pi} V(0+ b\e^{i\theta}) d\theta.$$ Hence, we obtain
$$V(0) \leq \frac{1}{2\pi}\int\limits_0^{2\pi} V(0+ b\e^{i\theta}) d\theta .$$

If the complex line intersects with \om\, not only at $0$, then using the fact that $\om $   is strictly convex in
a neighborhood of $0$ we can write $$\int\limits_0^{2\pi} V( b\e^{i\theta}) d\theta =
\int\limits_{\theta_1}^{\theta_2} u( b\e^{i\theta}) d\theta +
\int\limits_{\theta_2}^{\theta_1} v( b\e^{i\theta}) d\theta,$$ where $\pi / 2 <
\theta_1 < \pi $   and $\pi  < \theta_2 < 3\pi /2 $ . Now using the total differential of $u$ and $v$ around $0$ we have
\begin{eqnarray}
  u(z) &=& \sum_{i=1}^n u_{x_i}(0)x_i+ \sum_{i=1}^n u_{y_i}(0)y_i + \sum_{i=1}^n \eps_i x_i +\sum_{i=1}^n \eps'_i y_i, \\
  v(z) &=& \sum_{i=1}^n v_{x_i}(0)x_i+ \sum_{i=1}^n v_{y_i}(0)y_i + \sum_{i=1}^n \eta_i x_i +\sum_{i=1}^n \eta'_i y_i,
\end{eqnarray}
where $\eps_i,\, \eta_i \to 0$ as $x_i\to 0$ and $\eps'_i,\, \eta'_i\to 0$ as $y_i\to 0$. We remark
  that all first order partial derivatives of $u$ and $v$  vanish except the partial derivatives with respect to $x_1$.
This  is due to the facts that $(1,0,\dots,0)$   is the outer normal and  the rest of the standard basis elements lie in the tangent
space and  that $u$ and $v$ are defining functions for the boundary of \om. By defining  $A:=u_{x_1}(0)  = \frac{\p
u}{\p n}(0) $ and $B:=v_{x_1}(0) = \frac{\p v}{\p n}(0) $, we note that $ A\leq B$. Thus, we have
$$\int\limits_0^{2\pi} V( b\e^{i\theta}) d\theta =
\int\limits_{\theta_1}^{\theta_2} u( b\e^{i\theta}) d\theta +
\int\limits_{\theta_2}^{\theta_1} v( b\e^{i\theta}) d\theta,$$
which equals to
$$\int\limits_{\theta_1}^{\theta_2}
         \rho\left(u_{x_1}(0)  (a_1\cos(\theta)) +  \sum_{i=1}^n \eps_i(a_i\cos(\theta)) +
         \sum_{i=1}^n \eps'_i(b_i\sin(\theta))\right) d\theta $$
$$+ \int\limits_{\theta_2}^{\theta_1}\rho\left(v_{x_1}(0) \rho (a_1\cos(\theta)) +
\sum_{i=1}^n \eta_i (a_i\cos(\theta)) +\sum_{i=1}^n \eta'_i(b_i\sin(\theta))\right) d\theta ,$$
which is greater than or equal to$$ M:=(A-B)\rho [\sin(\theta_2)-\sin(\theta_1)]
-2\pi\rho[\sup(\eps_i, \eta_i, \eps'_i,\eta'_i)][\max_{i=1,\dots,n}\{|a_i|,|b_i|\}.
$$
Note that $M$ is positive by the facts that $A-B \leq 0 $, that $\sin(\theta_2)-\sin(\theta_1)<0$, that
$\eps_i,\, \eta_i \to 0$ as $x_i\to 0$ and $\eps'_i,\, \eta'_i\to 0$ as $y_i\to 0$.
So $M\geq 0 =V(0)$. Hence,  $V$ satisfies the integral inequality on each complex line. Therefore, $V$  is plurisubharmonic.
\end{proof}

\begin{proof}[Proof of Theorem \ref{SupportStrPsConvex}] Since \om \, is a strictly pseudoconvex domain, there exists a twice  continuously differentiable strictly  plurisubharmonic defining function $\rho$ defined in a neighborhood of the closure of \om. By Theorem 1.1 of  \cite{GuanRegularityGreenFunction},  the  extremal function $V_K$ is in $\mathcal{C}^{1,1}(\cn\setminus \om)$. Now we
define $\rho_\varepsilon := \varepsilon \rho$. Then, there exists $\varepsilon_0
>0$ that $\frac{\p \rho_{\varepsilon_0}}{\p n} < \frac{ \p V_K}{\p n}$ on $\p\om$.

We define $Q := \rho_{\varepsilon_0}$. Now consider
\begin{equation}
V :=\left\{\begin{array}{ccc}
Q & \quad \, \text{ on } &  \om ,                       \\
V_K & \quad \, \text{ on }  & \cn \setminus \om.   \\
\end{array}\right.
\end{equation}

By Theorem \ref{GluePshSmooth},  $V $ is plurisubharmonic and it is in Lelong class ${L}^+$. Since  $V=Q$ on $K$ and $V_K$ is maximal on $\cn \setminus \om$, by
domination principle,  we have $V_{K,Q}=V$ and the support $\supp (dd^{c} V_{K,Q})^n = \overline{\om} $ by definition of $Q$.
\end{proof}

The following observation gives us a way  of obtaining global subharmonic functions in Lelong class $L^+$ from a harmonic function defined in the unit disc (denoted by $\triangle$) of $\mathbb{C}$.

\begin{rem}\label{DiscHarmonic} Let $h$ be a harmonic function in the unit disc of $\mathbb{C}$ which is in
$C^{1}(\overline{\triangle})$. If $|\frac{\p h}{\p n}|\leq 1/2$, then the function defined by
\begin{equation}\label{GlueHarmonic}
g(z)=\left\{ \begin{array}{ll}
           h(z), & \hbox{ if } z\in \overline{\triangle}; \\
         h(\frac{1}{\bar{z}})+ \log|z| , & \hbox{ otherwise},\\
         \end{array}
       \right.
\end{equation}
is subharmonic and in $L^+$, where $n$ is the outer normal.  From this we can easily obtain a family of examples in the disc which is a special case of a result of Varju and Totik \cite{TotikVarju}.
\end{rem}

This observation gives us a way of constructing nonconstant weights on $\p\triangle$ such that the support of the weighted extremal measure is $p\triangle$.

Let $K$ be  the boundary of a bounded domain \om\, with smooth boundary. Our goal is  to find a continuous weight, $Q$, such that the
support of $(dd^{c} V_{K,Q})^n$ is the whole boundary. Note that  if \om\, is strictly pseudoconvex, then constant weights are sufficient for this purpose. However, we want
to find examples of nonconstant weights satisfying this condition. One may hope that if $V_{K,Q}=Q $ on $\p\om$, then
the support of the \mam is $\p\om$.

The following example shows that for a given nonconstant weight $Q$ on $\p\om=K$, the support  of $(dd^{c} V_{K,Q})^n$ is not
necessarily  all of $\p\om$. However, $V_{K,Q}=Q$ on $K$.

\begin{ex}
Let $K$ be the boundary of a bounded domain with smooth boundary. Let $L$ be a proper subset of $K$ which is regular
and the  polynomially convex hull of $L$ is a proper subset of  $K$. By regularity of $L$, we have  $V_L$ is continuous. Moreover, it is not constant on $K$ as it is not the polynomially convex hull of $K$.

We define $Q: K\to \mathbb{R}$ by  $Q(z)=V_L(z)$. Now $ V_{K,Q} = V_L$ whose
\mam is supported on  $L$, which is a proper subset of $K$.
\end{ex}

The next example shows that any closed ball can be obtained as the support of a Monge-{A}mp\`ere measure. Although this example is well known, we include it here for the sake of completeness.

\begin{ex}Let $K:= \overline{B(0,R)}$. We define $Q(z)=A(|z|^2-R^2)$ where $2AR\leq 1$. Then by Theorem \ref{GluePshSmooth}, we get
$$V_{K,Q}(z)=   \left\{
         \begin{array}{ll}
           A(|z|^2-R^2), & \hbox{ if } |z|\leq R; \\
          \log|z|-\log R, & \hbox{ if } |z|>R.
         \end{array}
       \right.
$$
It is clear that the support of the \mam is $K$.
\end{ex}

The following example shows that the finite union of concentric spheres can be obtained as the support of a weighted equilibrium measure.
\begin{ex}\label{Spheres}
Let $K_m=\overline{B(0,r_m)}$ for some $r_m>0$. We will inductively define weights on $K_m$ such that the weighted equilibrium
measure is supported on the concentric spheres, $\bigcup_{i=1}^m \p B(0,r_i)$ for $ m \geq 1$ and $0 < r_1 < r_2 < \dots <
r_m $.

Let $m=1$, we define $Q_1(z)=\log r_1$. Clearly, $V_{K_1,Q_1}(z)=\max(\log|z|,\log r_1)$ and $(dd^c V_{K_1,Q_1})^n$
is supported on the sphere of radius $r_1$ with center origin.

For $m \geq 2$, we define $Q_m=\frac{1}{2}V_{K_{m-1},Q_{m-1}}|_{K_m}$. We remark that each $Q_m$ is continuous and
$K_m$ is locally regular, hence, each of  $V_{K_m,Q_{m}}$ is continuous by Proposition 2.13 of \cite{SiciakExtremal1}. Since each weight function and the
weighted extremal function are radial,   we can write  $V_{K_m,Q_{m}}(z)=V_{K_m,Q_{m}}(|z|)$.  We will show that
$$V_{K_m,Q_{m}}(z)=\left\{ \begin{array}{ll}
          \frac{1}{2}V_{K_{m-1},Q_{m-1}}(z), & \hbox{ if } |z|\leq r_m; \\
         \log|z|+A_m , & \hbox{ if } |z|>r_m,
         \end{array}
       \right.$$
where $A_m= \frac{1}{2}V_{K_{m-1},Q_{m-1}}(r_m)-\log r_m$.

If the function $v_m$, on the right hand side, is plurisubharmonic then it is maximal outside of $K_m$. In addition, using the fact that $v_m$ equals to $Q_m$ on $K_m$, we have  $V_{K_m,Q_{m}}=v_m$  by the domination principle. Thus, it is enough to show that the function $v_m$  is plurisubharmonic.

For $m=2$, we have  $\frac{1}{2}V_{K_{1},Q_{1}}=\frac{1}{2}\max(\log|z|,\log r_1)$.
Since
$$\frac{\p \left\{\frac{1}{2} \max(\log|z|,\log r_1)\right\}}{\p n}< \frac{\p \left\{\log|z|+A_2\right\}}{\p n}$$ on $\p B(0,r_2)$,  by  Theorem \ref{GluePshSmooth}, we have that $v_2$ is plurisubharmonic.

For $m>2$, by induction hypothesis we have $V_{K_m,Q_{m}}(z)=\log|z|+ A_{m-1}$ on $B(0,r_m)\setminus B(0,r_{m-1})$. Hence, again by Theorem \ref{GluePshSmooth},  we obtain that $v_m$ is plurisubharmonic.

Next, we show that $\supp (dd^c V_{K_m,Q_m})^n=\bigcup_{i=1}^m \p B(0,r_i)$.

For $m=1$, this is obvious.
For $m\geq2$,
$$V_{K_m,Q_{m}}(z)=\left\{ \begin{array}{ll}
         \frac{1}{2}V_{K_{m-1},Q_{m-1}}, & \hbox{ if } |z|\leq r_m; \\
         \log|z|+A_m , & \hbox{ if } |z|>r_m.
         \end{array}
       \right.$$
Clearly, $V_{K_m,Q_{m}}$ is maximal on $\cn \setminus K_m$. Hence, $\supp (dd^c V_{K_m,Q_m})^n\subset K_m$.
Since $V_{K_m,Q_{m}}(z)=\frac{1}{2}(\log|z|+A_{m-1})$ on $B(0,r_m)\setminus B(0,r_{m-1})$, we have   $$\p B(0,r_m)\subset \supp (dd^c V_{K_m,Q_m})^n.$$
On $B(0,r_m)$, we have $V_{K_m,Q_{m}}= \frac{1}{2}V_{K_{m-1},Q_{m-1}}$. By induction assumption, we have
$$\supp (dd^c V_{K_{m-1},Q_{m-1}})^n= \bigcup_{i=1}^{m-1} \p B(0,r_i).$$ Therefore, we have $\supp (dd^c V_{K_m,Q_m})^n=\bigcup_{i=1}^m \p B(0,r_i)$.

Note that in this example, we can take $K_m=\bigcup_{i=1}^m \p B(0,r_i)$ and  similar arguments will give the same conclusion.
\end{ex}

The next example shows that shells, i.e.,  the difference of two concentric balls,  can be obtained as the support of the \mam of a
weighted extremal function.

\begin{ex}\label{SupportShellMeasure} Let $K=\overline{B(0,R)} \setminus B(0,r)$ where $r<R$.
Let $$Q(z)= \frac{1}{R-r}\left(|z|-r\log|z|-r+r \log R\right).$$ Later we will verify that
$$V_{K,Q}(z)=\left\{ \begin{array}{ll}
           0, & \hbox{ if } |z|<r; \\
         \frac{1}{R-r}\left(|z|-r\log|z|-r+r \log R\right), & \hbox{ if } r\leq |z| \leq R;\\
         \log|z|-\log R +\frac{1}{R-r}(R-r\log R-r+r \log R), & \hbox{ if } |z|>R
         \end{array}
       \right.
$$
and that the support of the \mam of the weighted extremal function is the shell   $K=\overline{B(0,R)} \setminus B(0,r)$.
\end{ex}

Our next goal is  to show that many radially symmetric compact sets can be obtained as the support of the \mam of a
weighted extremal function $V_{K,Q}^\ast$, for some admissible continuous weight $Q$.

To obtain results on radially symmetric compact sets and in order to verify Example \ref{SupportShellMeasure},   we
recall some facts about radially symmetric plurisubharmonic functions and  the representation of these functions in terms of their Monge-{A}mp\`ere measures   from Persson's thesis
\cite{PerssonSymetric}. We note that this work is not published anywhere else.

\begin{prop}\cite[Proposition 3.1]{PerssonSymetric}
Let $u(z)=\widetilde{u}(\log|z|^2)$ be an upper semicontinuous function. Then $u$ is plurisubharmonic on $B(0,e^R)$ if and
only if $\tilde{u}$ is increasing  on $[-\infty, R)$  and convex on $(-\infty, R)$, for $-\infty< R \leq \infty$.
\end{prop}

Persson  computes the entries of the  complex Hessian matrix of a radially symmetric  plurisubharmonic function $u(z):=
\widetilde{u}(\log|z|^2) \in \mathcal{C}^2$ and obtains
\begin{equation}
\frac{\p^2 u}{\p z_i\p\bar{z}_j}
=\delta_{ij}\frac{\widetilde{u}'(\log|z|^2)}{|z|^2}+\bar{z}_i z_j
\frac{\widetilde{u}''(\log|z|^2)-\widetilde{u}'(\log|z|^2)}{|z|^4}.
\end{equation}
Using linear  algebra, he obtains the \mam of $u$ as
$$(dd^c u)^n = \frac{\left(\widetilde{u}'(\log|z|^2)\right)^{n-1}
\widetilde{u}''(\log|z|^2) }{|z|^{2n}}dV(z),$$
where $dV$ is the standard volume form on \cn.
By direct integration, he obtains the following representation theorem.
\begin{thm}\cite[Corollary 3.4]{PerssonSymetric}\label{PerssonRepresentationTheorem}
Let $u$ be a radial plurisubharmonic function bounded from below,  then we have the
following representation formula.
\begin{equation}
u(z)=u(0)+\int\limits_{0<|\zeta|<|z|}
\frac{|\zeta|^{-2n}}{n\omega_{2n}}\left(\frac{(dd^cu)^n (B(0,|\zeta|))}{4^n
n!\omega_{2n} }\right)^{1/n}dV(\zeta).
\end{equation}
Conversely, if $\mu$ is a radially symmetric nonnegative measure with compact support satisfying
\begin{equation}\label{RadialCOndition}
\int\limits_{B(0,r)} |\zeta|^{-2n}\left(\mu (B(0,|\zeta|)) \right)^{1/n}dV(\zeta)
<\infty ,
\end{equation} then
\begin{equation}
u(z)=\int\limits_{0<|\zeta|<|z|}
\frac{|\zeta|^{-2n}}{n\omega_{2n}}\left(\frac{\mu (B(0,|\zeta|))}{4^n
n!\omega_{2n} }\right)^{1/n}dV(\zeta)
\end{equation}
defines a  radially symmetric plurisubharmonic function bounded from below and $(dd^cu)^n =\mu$.
\end{thm}

As a by-product of the proof of Theorem 3.3 of \cite{PerssonSymetric}, we   get the following one dimensional integral
representation  for radially symmetric plurisubharmonic functions which are bounded from below \cite[page
15]{PerssonSymetric}.
\begin{equation}\label{RadialMeasurePsh}
 u(z)=  u(0)+ \int\limits_0^{|z|} \frac{2}{t}\left(\frac{\mu(B(0,t))}{4^n n!
 \omega_{2n}}\right)^{1/n}dt.
\end{equation}
Given a radially symmetric compact set $K$, we are going to construct an appropriate positive radially  symmetric measure
whose support is $K$ and which satisfies \eqref{RadialCOndition}. Then we will use  the representation \eqref{RadialMeasurePsh}
to construct a radially plurisubharmonic function in the Lelong class $L^+$ with \mam supported on $K$.

Given a radially symmetric $\mu$ satisfying \eqref{RadialCOndition}, we make the following observations. Here, we use the notation $f(t):=\mu(B(0,t))$.
\begin{enumerate}
\item $f$ is a nondecreasing function of $t$.
\item $f(t)$  is constant for $t>T_0$ for some $T_0>0$ if and only if $\mu$ has compact support.
\item Let $\mu$  be a  measure  with compact support, i.e.,  $f(t)$  is constant for $t>T_0$, then
$u\in {L^+}$ if and only if $2\left(\frac{\mu(B(0,T_0))}{4^n n! \omega_{2n}}\right)^{1/n}=1$,
 i.e.,  $\mu(B(0,T_0))=(2\pi)^n$. This condition  comes from  equation \eqref{RadialMeasurePsh}  and it is
 a necessary condition for $u$ to be in ${L^+}$.
\end{enumerate}

As an immediate application of the above formula, we verify   Example \ref{SupportShellMeasure}, where we obtain  shells
 as support of the weighted equilibrium measure.

 We define the following radial measure
$$f(t):=\mu(B(0,t))=\left\{ \begin{array}{ll}
           0, & \hbox{ if } t<r; \\
         \left(\frac{2\pi}{R-r}(t-r)\right)^n, & \hbox{ if } r<t<R;\\
         (2\pi)^n , & \hbox{ if } t>R.
         \end{array}
       \right.
$$
Assuming $u(0)=0$, we obtain the corresponding function defined  by \eqref{RadialMeasurePsh} as:
$$u(z)=\left\{ \begin{array}{ll}
           0, & \hbox{ if } |z|<r; \\
         \frac{1}{R-r}\left(|z|-r\log|z|-r+r \log R\right), & \hbox{ if } r<|z|<R;\\
         \log|z|-\log R +\frac{1}{R-r}(R-r\log R-r+r \log R), & \hbox{ if } |z|>R,
         \end{array}
       \right.
$$
which is a continuous plurisubharmonic function. Note that $u$ is maximal off $K$ and
$u\in{L}^+$. Since   $Q=u|_K$, we obtain $V_{K,Q}=u$.

In fact, we get a similar construction by  taking any continuous nondecreasing function $f(t)$ of the form
\begin{equation}\label{AnnuliMeasure}f(t):=\mu(B(0,t))=\left\{ \begin{array}{ll}
           0, & \hbox{ if } t<r; \\
         g(t), & \hbox{ if } r<t<R;\\
         (2\pi)^n , & \hbox{ if } t>R,
         \end{array}
       \right.
\end{equation}
where $g(t)$  is a strictly  increasing function such that $g(r)=0$ and
$g(R)=(2\pi)^n$.

We remark that we can obtain the case of  $m=1$ in Example \ref{Spheres} by   using a measure $\mu$ of the form
\begin{equation}\label{SphereMeasure}
\mu(B(0,t))=\left\{ \begin{array}{ll}
           0, & \hbox{ if } t<r_1; \\
                  (2\pi)^n , & \hbox{ if } t\geq r_1.
         \end{array}
       \right.
\end{equation}
Thus, by \eqref{RadialMeasurePsh}, we obtain $V_{K,Q}(z)=\max\{0, \log |z|-\log r_1\}$.

A similar construction works with a compact set  $K$ which is a countable  union of spheres and closed  shells.
Let $K:=\bigcup\limits_{i=1}^\infty K_i $ where each $K_i$ is either a sphere or an annulus. We define a measure $\mu:=
\sum\limits_{i=1}^\infty \frac{\mu_i}{2^i}$ where each $\mu_i$ is supported on each $K_i$. If $K_i$ is a sphere, then
$\mu_i$ is of the form \eqref{SphereMeasure} and if $K_i$ is an annulus, then  $\mu_i$ is of the form \eqref{AnnuliMeasure}.  By using \eqref{RadialMeasurePsh} we obtain a \psh function $u\in L^+$ such that the support of the \mam of $u$ is $K$. Hence,
by defining $Q=u|_K$, we obtain the desired result.

\section{Open Problems}
In this section, we list some open problems related to supports of extremal measures in connection with supports of Monge-{A}mp\`ere measures of plurisubharmonic functions in Lelong class.

\begin{oprob}
Under which conditions weighted extremal measures $(dd^c V_{K,Q}^\ast)^n$ and $(dd^c U_{K,Q,\om}^\ast)^n$ are mutually absolutely continuous? In particular, if their supports are equal, are they mutually absolutely continuous?
\end{oprob}

Note that  Example \ref{Example1} and Example \ref{Example2} show that continuous weights and plurisubharmonic weights are not enough to have mutual absolute continuity of the extremal functions. The conclusion holds for constant weights obviously.

\begin{oprob}
Find conditions on $K$ and $Q$ such that $V_{K,Q}=Q$ on $K$. Note that in this case $Q$ should be continuous plurisubharmonic in the interior of $K$.
\end{oprob}

\begin{oprob}
Under which conditions a compact set $K\subset \cn$ can be written as the support of a weighted extremal function? Namely, which sets are the support of the \mam of a continuous plurisubharmonic function in Lelong class, $L^+$?
\end{oprob}

Note that, in this case the set $K$ cannot be pluripolar at any point.

\begin{oprob}
Which (compactly supported) measures  are the Monge-{A}mp\`ere measures of continuous plurisubharmonic functions in Lelong class?
\end{oprob}

The relation between this problem and the previous one is: If there exists $u\in L^+ \cap \mathcal{C}(\cn)$ such that $\supp (dd^c u)^n=K$ then we can define $Q=u|_K$which is an admissible weight and by the domination principle $V_{K,Q}=u$.

This question is equivalent to finding the range of $T: L^+ \cap \mathcal{C}(\cn) \to \mathcal{PL}(\cn)$, where  $T(u)=(dd^c u)^n$ and $\mathcal{PL}(\cn) $ is the set of Borel measures with compact support and total mass $(2\pi)^n$. In this case if $\mu$ is in the range of $T$, then $\mu$ has no mass on pluripolar sets.

Note that Guedj and Zeriahi showed that if $\mu$ is a measure which puts no mass on pluripolar sets and has total mass $(2\pi)^n$, then there exists $u\in L^+$ such that $(dd^c u)^n=\mu$. See \cite{GuedjZeriahi}.

Also note that for $n=1$, Arsove answered this question in terms of density of the measure. However, again the notion of  logarithmic potentials is used. See \cite{Arsove} for details.

\textbf{Acknowledgements.}{ I thank Professor Norman Levenberg for his excellent support and help.}

\bibliographystyle{amsplain}

\providecommand{\bysame}{\leavevmode\hbox to3em{\hrulefill}\thinspace}
\providecommand{\MR}{\relax\ifhmode\unskip\space\fi MR }
\providecommand{\MRhref}[2]{%
  \href{http://www.ams.org/mathscinet-getitem?mr=#1}{#2}
}
\providecommand{\href}[2]{#2}

\end{document}